\pdfoutput=1

\documentclass{amsart}

\usepackage[utf8]{inputenc}
\usepackage[T1]{fontenc}
\usepackage{lmodern}
\usepackage[ngerman,english]{babel}
\usepackage{graphicx}
\usepackage{mathtools}

\newtheorem{theorem}{Theorem}
\newtheorem{lemma}[theorem]{Lemma}

\theoremstyle{definition}
\newtheorem{definition}[theorem]{Definition}

\theoremstyle{remark}
\newtheorem{remark}[theorem]{Remark}

\newcommand*{\RR}{\ensuremath{\mathbb R}}
\newcommand*{\QQ}{\ensuremath{\mathbb Q}}
\newcommand*{\NN}{\ensuremath{\mathbb N}}
\newcommand*{\II}{\ensuremath{\mathcal I}}
\newcommand*{\DD}{\ensuremath{\mathcal D}}
\newcommand*{\LL}{\ensuremath{\mathfrak L}}

\allowdisplaybreaks[3]

\title[Normal numbers from equidistributed sequences]{Construction of normal numbers \\ with respect to Generalized Lüroth Series \\ from equidistributed sequences}
\author{Max Aehle}
\address{Max Aehle \\ Max Planck Institute for Mathematics \\ Vivatsgasse~7 \\ D-53111~Bonn}
\curraddr{Faculty of Mathematics \\ Technische Universität Kaiserslautern \\ P.\,O.~Box~3049 \\ D-67653~Kaiserslautern}
\email{aehle@rhrk.uni-kl.de}
\author{Matthias Paulsen}
\address{Matthias Paulsen \\ Max Planck Institute for Mathematics \\ Vivatsgasse~7 \\ D-53111~Bonn}
\curraddr{Department Mathematisches Institut \\ Universität München \\ Theresienstr.~39 \\ D-80333~München}
\email{matthias.paulsen@campus.lmu.de}

\begin{document}
\begin{abstract}
Generalized Lüroth series generalize $b$-adic representations as well as
Lüroth series.
Almost all real numbers are normal, but it is not easy to construct one.
In this paper, a new construction of normal numbers with respect to Generalized Lüroth Series (including those with an infinite digit set) is given. Our method concatenates the beginnings of the expansions of an arbitrary equidistributed sequence.
\end{abstract}

\maketitle

\section{Introduction}\label{sect:intro}
In 1909, Borel proved that the $b$-adic representations of almost all real numbers are normal for any integer base $b$. However, only a few simple examples for normal numbers are known, like the famous Champernowne constant \cite{champ}
\[ 0\,.\,1\,2\,3\,4\,5\,6\,7\,8\,9\,10\,11\,12\,13\,14\,15\,16\,17\,18\,19\,20\,\dots \;. \]
Other constants as $\pi$ or $\mathrm e$ are conjectured to be normal.
The concept of normality naturally extends to more general number expansions, including Generalized Lüroth Series \cite[pp.~41--50]{def-gls}. For example, the numbers $1/\mathrm e$ and $\frac12\sqrt3$ are not normal with respect to the classical Lüroth expansion that was introduced in \cite{lueroth}.

In his Bachelor thesis, Boks \cite{boks} proposed an algorithm that yields a normal number with respect to the Lüroth expansion, but he did not complete his proof of normality. Some years later, Madritsch and Mance \cite{mu-normal} transferred Champernowne's construction to any invariant probability measure. However, their method is rather complicated and does not reflect that digit sequences represent numbers. Vandehey \cite{vandehey} provided a much simpler construction, but for finite digit sets only, in particular, not for the original Lüroth expansion.

Like classical Lüroth expansions, continued fraction expansions are number representations with an infinite digit set. Adler, Keane and Smorodinsky \cite{cfrac} showed that the concatenation of the continued fractions of
\begin{equation}
\frac12,\;\frac13,\;\frac23,\;\frac14,\;\frac24,\;\frac34,\;\frac15,\;\frac25,\;\frac35,\;\frac45,\;\dots \label{eq:cfrac}
\end{equation}
leads to a normal number.

This motivates our approach for the construction of normal numbers with respect to Generalized Lüroth Series (GLS) from equidistributed sequences.
The idea is to look at the expansions of some equidistributed sequence with respect to a given GLS. In general, these expansions are infinite. Hence, in order to concatenate them, we have to trim each digit sequence to a certain length.

The paper is organized as follows: In Sections \ref{sect:equi} and \ref{sect:gls}
we recall the notions of equidistributed sequences and GLS respectively,
then in Section~\ref{sect:constr} we describe our construction, finally,
in Section~\ref{sect:proof} we prove that the constructed number is normal.

\section{Equidistributed Sequences}\label{sect:equi}
Recall that a sequence $(a_1,a_2,\dots)$ of numbers $a_j\in[0,1]$ is called equidistributed if for any interval $I\subseteq[0,1]$,
\begin{equation}
\lim_{n\to\infty}\frac{\#\{1\le j\le n:a_j\in I\}}n = |I| \;. \label{eq:def-equi}
\end{equation}
\begin{remark}\label{re:shift}
If $(a_1,a_2,\dots)$ is equidistributed, then $(a_k,a_{k+1},\dots)$ is
equidistributed for any $k\in\NN=\{1,2,\dots\}$.
\end{remark}
Many easy computable equidistributed sequences exist, e.\,g.
the sequence~\eqref{eq:cfrac} used by \cite{cfrac} and the sequence $(n\beta\mod1)_{n\in\NN}$ for any $\beta\in\RR\setminus\QQ$ \cite{uniform}.

It is convenient to introduce the concept of uniform equidistribution:
\begin{definition}\label{def:uniequi}
A sequence $(a_1,a_2,\dots)$ of numbers $a_j\in[0,1]$ is called \emph{uniformly equidistributed} if the convergence in equation~\eqref{eq:def-equi} is uniform in $I$, i.\,e. for all $\varepsilon>0$ there is an $N\in\NN$ such that for all
$n\ge N$ and any interval $I\subseteq[0,1]$,
\[ \left|\frac{\#\{1\le j\le n:a_j\in I\}}n - |I|\right| < \varepsilon \;. \]
\end{definition}
As it turns out, these two terms are equivalent.
This fact is commonly used without proving it. For clarity, we will sketch a simple proof here.
\begin{lemma}\label{lem:uniequi-equi}
Every equidistributed sequence is uniformly equidistributed.
\end{lemma}
\begin{proof}
Consider an equidistributed sequence $(a_1,a_2,\dots)$. Let $\varepsilon>0$.
Choose $k\in\NN$ such that $\frac1k<\varepsilon/3$. Set $E_i=\left[\frac{i-1}k,\frac ik\right)$ for $i=1,\dots,k-1$ and $E_k=\left[\frac{k-1}k,1\right]$.
For each $E_i$ there is an $N_i\in\NN$ with
\[ \left|\frac{\#\{1\le j\le n:a_j\in E_i\}}n - |E_i|\right| < \frac\varepsilon{3k} \]
for all $n\ge N_i$. Let $N=\max\{N_1,\dots,N_k\}$. For a given interval $I\subseteq[0,1]$,
consider the intervals $E_s,\dots,E_t$ ($1\le s\le t\le k$) which have a
non-empty intersection with $I$. Take an arbitrary integer $n\ge N$. We get
\begin{align*}
\frac{\#\{1\le j\le n:a_j\in I\}}n-|I| &\le \frac{\#\{1\le j\le n:a_j\in\bigcup_{i=s}^t E_i\}}n - \left|\bigcup_{i=s+1}^{t-1}E_i\right| \\
&\le \sum_{i=s}^t\left(\frac{\#\{1\le j\le n:a_j\in E_i\}}n-|E_i|\right)+|E_s|+|E_t| \\
&\le \sum_{i=1}^k\left|\frac{\#\{1\le j\le n:a_j\in E_i\}}n-|E_i|\right|+\frac2k \\
&< k\cdot\frac\varepsilon{3k}+2\varepsilon/3 = \varepsilon \;.
\end{align*}
Similarly, $-\varepsilon$ is a lower bound.
\end{proof}
We are interested in maps which preserve the equidistribution property of sequences.
\begin{definition}\label{def:equipres}
A map $T\colon[0,1]\to[0,1]$ is called \emph{equidistribution-preserving} if for every
equidistributed sequence $(a_1,a_2,\dots)$ the sequence $(Ta_1,Ta_2,\dots)$ is equidistributed.
\end{definition}

\section{Generalized Lüroth Series}\label{sect:gls}
Based on \cite[pp.~41--50]{def-gls}, we precise our notion of Generalized Lüroth Series.
\begin{definition}\label{def:gls}
Let $\DD\subseteq\NN$ be a non-empty set (called ``digit set''). Suppose that for each
$d\in\DD$ there is an interval $\II_d\subseteq[0,1]$ and a map $T_d\colon\II_d\to[0,1]$
with the following properties:
\begin{enumerate}
\item\label{it:union} $\displaystyle\bigcup_{d\in\DD}\II_d=[0,1]$;
\item\label{it:disjoint} the intervals $\II_d$ are pairwise disjoint;
\item if $|\II_d|>0$, then $T_d$ is either of the form
\[ T_d(x) = \frac1{|\II_d|}(x-\inf\II_d) \qquad\text{or}\qquad T_d(x) = 1-\frac1{|\II_d|}(x-\inf\II_d) \;. \]
\end{enumerate}
Define $T\colon[0,1]\to[0,1]$ by $T(x)=T_d(x)$ for $x\in\II_d$.

The tuple $\LL=(\DD,\{\II_d:d\in\DD\},T)$ is called a \emph{Generalized Lüroth Series (GLS)}. For $j\in\NN_0$ we denote the $j$-th digit of $x\in[0,1]$ with respect to $\LL$ by $\LL_j(x)$, i.\,e. the $d\in\DD$ such that $T^jx\in\II_d$.
We write $\LL_j^\ell(x)=(\LL_j(x),\LL_{j+1}(x),\dots,\LL_{j+\ell-1}(x))$.
\end{definition}
Indeed, the transformation of a GLS possesses the desired property.
\begin{lemma}\label{lem:gls-equipres}
Let $\LL=(\DD,\{\II_d:d\in\DD\},T)$ be a GLS. Then $T$ is equidistribution-preserving.
\end{lemma}
\begin{proof}
Let $(a_1,a_2,\dots)$ be a uniformly equidistributed sequence (by Lemma~\ref{lem:uniequi-equi}). Let $I\subseteq[0,1]$ and $\varepsilon>0$ be arbitrary. Because of properties (\ref{it:union}) and (\ref{it:disjoint}) of Definition~\ref{def:gls} we know $\sum_{d\in\DD}|\II_d|=1$,
i.\,e. there are intervals $\II_{d_1},\dots,\II_{d_k}$ with
\begin{equation}
\sum_{i=1}^k|\II_{d_i}|>1-\varepsilon/2 \;. \label{eq:1-eps/2}
\end{equation}
The set $[0,1]\setminus\bigcup_{i=1}^k\II_{d_i}$ is a union of pairwise disjoint intervals $H_0,\dots,H_{k'}$ with $k'\le k$. Let $J_i=(T^{-1}I)\cap\II_{d_i}$ for $i=1,\dots,k$. Definition~\ref{def:gls} implies that $J_1,\dots,J_k$ are intervals.
It is a well known fact that $T$ preserves the
Lebesgue measure. Hence
\begin{equation}
\sum_{i=1}^k|J_i| \le |I| \le \sum_{i=1}^k|J_i|+\sum_{i=0}^{k'}|H_i| \;. \label{eq:|I|}
\end{equation}
Using Definition~\ref{def:uniequi}, there exists an integer $N\in\NN$ such that for
all $n\ge N$ we have
\begin{equation}
\left|\frac{\#\{1\le j\le n:a_j\in X\}}n-|X|\right|<\frac\varepsilon{4k+2} \;, \label{eq:eps/4k+2}
\end{equation}
whereas $X\in\{J_1,\dots,J_k,H_0,\dots,H_{k'}\}$.
For the upper bound we get for every $n\ge N$
\begin{align*}
 \frac{\#\{1\le j\le n:Ta_j\in I\}}n-|I|
&\le \sum_{i=1}^k\frac{\#\{1\le j\le n:a_j\in J_i\}}n \\*
&\phantom{{}\le{}} \qquad {}+\sum_{i=0}^{k'}\frac{\#\{1\le j\le n:a_j\in H_i\}}n-|I| \\
&\overset{\eqref{eq:|I|}}\le \sum_{i=1}^k\left(\frac{\#\{1\le j\le n:a_j\in J_i\}}n-|J_i|\right)+\sum_{i=0}^{k'}|H_i| \\*
&\phantom{{}\le{}} \qquad {}+\sum_{i=0}^{k'}\left(\frac{\#\{1\le j\le n:a_j\in H_i\}}n-|H_i|\right) \\
&\overset{\mathclap{\eqref{eq:1-eps/2},\eqref{eq:eps/4k+2}}}< k\cdot\frac\varepsilon{4k+2}+\varepsilon/2+(k'+1)\cdot\frac\varepsilon{4k+2} \le \varepsilon \;.
\intertext{For the lower bound we similarly obtain}
|I|-\frac{\#\{1\le j\le n:Ta_j\in I\}}n
&\le |I|-\sum_{i=1}^k\frac{\#\{1\le j\le n:a_j\in J_i\}}n \\
&\overset{\eqref{eq:|I|}}\le \sum_{i=0}^{k'}|H_i|+\sum_{i=1}^k\left(|J_i|-\frac{\#\{1\le j\le n:a_j\in J_i\}}n\right) \\
&\overset{\mathclap{\eqref{eq:1-eps/2},\eqref{eq:eps/4k+2}}}< \varepsilon/2+k\cdot\frac\varepsilon{4k+2}<\varepsilon \;. \qedhere
\end{align*}
\end{proof}

With respect to a GLS $\LL$, a number $x\in[0,1]$ is normal if for any
block $b=(b_1,\dots,b_r)\in\DD^r$, we have
\begin{equation}
\lim_{n\to\infty}\frac{\#\{0\le j\le n-1:\LL_j^r(x)=b\}}n = \prod_{i=1}^r|\II_{b_i}| \;. \label{eq:def-normal}
\end{equation}

\section{Construction}\label{sect:constr}
Fix a GLS $\LL=(\DD,\{\II_d:d\in\DD\},T)$ and an equidistributed sequence
$(a_1,a_2,\dots)$. We recursively define a strictly increasing
sequence $(c_0,c_1,c_2,\dots)$ of indices as follows. Let $c_0=0$.
Assume $c_0,\dots,c_\ell$ are already chosen for some $\ell\in\NN_0$.
For $i=0,\dots,\ell$ the sequences $(T^ia_{c_i+1},T^ia_{c_i+2},\dots)$ are uniformly equidistributed by Remark~\ref{re:shift},
Lemma~\ref{lem:uniequi-equi} and Lemma~\ref{lem:gls-equipres}. Therefore we can choose a $c_{\ell+1}>c_\ell$ such that for all integers $n\ge c_{\ell+1}$ and $0\le i\le\ell$ and any interval $I\subseteq[0,1]$ we have
\begin{equation}
\left|\frac{\#\{c_i+1\le j\le n:T^ia_j\in I\}}{n-c_i}-|I|\right| < \frac1{\ell+1} \;. \label{eq:c}
\end{equation}
Finally, we define $z\in[0,1]$ such that its expansion with respect to $\LL$
is the concatenation of
$\LL_0^{l(1)}(a_1)$, $\LL_0^{l(2)}(a_2)$, $\LL_0^{l(3)}(a_3)$, etc., whereas
$l(j)$ is chosen such that $c_{l(j)-1}+1\le j\le c_{l(j)}$.
In particular, we have $l(j)\le j$ as well as $l(j)\to\infty$ for $j\to\infty$.

\section{Proof of Normality}\label{sect:proof}
\begin{figure}
\centering\includegraphics[width=\linewidth]{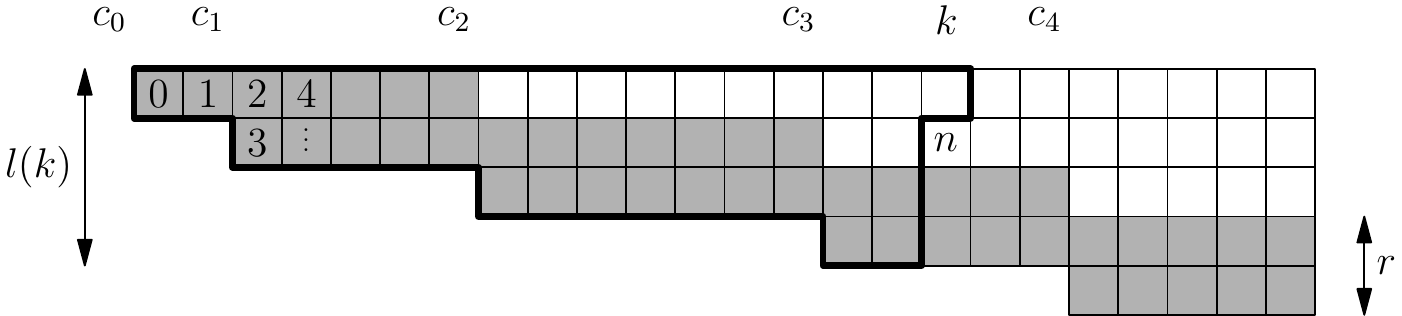}
\caption{Visualization of the construction of $z$. The $m$-th digit of $z$ in the $i$-th row ($i=0,1,\dots$) and $j$-th column ($j=1,2,\dots$) is $\LL_m(z)=\LL_i(a_j)$.}\label{fig:quad}
\end{figure}
\begin{theorem}
The number $z$ is normal with respect to $\LL$.
\end{theorem}
\begin{proof}
Let $b=(b_1,\dots,b_r)\in\DD^r$ be a block of digits.
Define $I\subseteq[0,1]$ as the set of all $x\in[0,1]$ which fulfill $\LL_0^r(x)=b$. It follows that $I$ is an interval of length $\prod_{i=1}^r|\II_{b_i}|$.
We want to verify equation~\eqref{eq:def-normal}.
Fix an arbitrary $n\in\NN$.
Assume that the $(n-1)$-th digit of $z$ (i.\,e. $\LL_{n-1}(z)$) occurs within
$\LL_0^{l(k)}(a_k)$ for some $k\in\NN$.
Obviously, the number of digit positions $m\in\{0,\dots,n-1\}$ which belong
to an $\LL_i(a_j)$ for a certain $i=0,\dots,l(k)-1$ is bounded above
by $k-c_i$, thus
\begin{equation}
n \le \sum_{i=0}^{l(k)-1}(k-c_i) \le \sum_{i=0}^{l(k)-2}(k-c_i)+k \;. \label{eq:n-le}
\end{equation}
Similarly, it is bounded below by $k-1-c_i$, hence
\begin{equation}
n \ge \sum_{i=0}^{l(k)-1}(k-1-c_i) \ge \sum_{i=0}^{l(k)-2}(k-c_i-1) \label{eq:n-ge} \;.
\end{equation}
Counting the number of occurrences of $b$ in each column separately (see Figure~\ref{fig:quad}), we obtain
\begin{align*}
&\phantom{{}={}} \frac{\#\{0\le j\le n-1:\LL_j^r(z)=b\}}n-|I| \\
&\le \sum_{j=1}^k\frac{\#\{0\le i\le l(j)-r-1:\LL_i^r(a_j)=b\}+r}n-|I| \\
&\le \sum_{j=1}^k\frac{\#\{0\le i\le l(j)-2:T^ia_j\in I\}}n-|I|+\frac kn\cdot r \\
&= \sum_{i=0}^{l(k)-2}\frac{\#\{c_i+1\le j\le k:T^ia_j\in I\}}n-|I|+\frac kn\cdot r \\
&\overset{\eqref{eq:n-ge}}\le \sum_{i=0}^{l(k)-2}\frac{k-c_i}n\left(\frac{\#\{c_i+1\le j\le k:T^ia_j\in I\}}{k-c_i}-|I|\right)+\frac{l(k)-1}n\cdot|I|+\frac kn\cdot r \\
&\overset{\eqref{eq:c}}\le \sum_{i=0}^{l(k)-2}\frac{k-c_i}n\cdot\frac1{l(k)-1}+\frac{l(k)-1}n\cdot|I|+\frac kn\cdot r \\
&\overset{\eqref{eq:n-ge}}\le \frac{n+l(k)-1}{n(l(k)-1)}+\frac{l(k)-1}n\cdot|I|+\frac kn\cdot r \le \frac1{l(k)-1}+\frac1n+\frac kn\cdot(|I|+r) \;.
\intertext{In a similar manner, we establish a lower bound:}
& \phantom{{}={}} |I|-\frac{\#\{0\le j\le n-1:\LL_j^r(z)=b\}}n \\
&\le |I|-\sum_{i=0}^{l(k)-2}\frac{\#\{c_i+1\le j\le k-1:T^ia_j\in I\}}n+\frac{k-1}n\cdot r \\
&\overset{\eqref{eq:n-le}}\le \sum_{i=0}^{l(k)-2}\frac{k-c_i-1}n\left(|I|-\frac{\#\{c_i+1\le j\le k-1:T^ia_j\in I\}}{k-c_i-1}\right) \\*
&\phantom{{}\le{}} \qquad {}+\frac{l(k)-1+k}n\cdot|I|+\frac{k-1}n\cdot r \\
&\overset{\mathclap{\eqref{eq:c},\eqref{eq:n-ge}}}\le \frac1{l(k)-1}+\frac{l(k)-1+k}n\cdot|I|+\frac{k-1}n\cdot r \le \frac1{l(k)-1}+\frac kn\cdot(2|I|+r) \;.
\end{align*}
For $n\to\infty$, we have $k\to\infty$ and thus $l(k)\to\infty$. Using the Cauchy limit theorem, we conclude
\[ \frac nk\ge\frac{0+l(1)+l(2)+\dots+l(k-1)}k\to\infty \;. \]
It follows that
\[ \lim_{n\to\infty}\left|\frac{\#\{0\le j\le n-1:\LL_j^r(z)=b\}}n-|I|\right|=0 \;. \]
This completes the proof.
\end{proof}

\section{Open Problems}\label{sect:open}
It would be more straightforward and closer to \cite{cfrac}, if the trimming of the expansions could be omitted by using equidistributed sequences of rational numbers with a finite expansion. Lüroth himself proved that rational numbers always have a finite or periodic Lüroth expansion \cite{lueroth}, but he did not find a way to distinguish between these two possibilities. Some algebraic manipulation shows that fractions of the form $\frac a{2^k}$ (i.\,e. the dyadic rationals) have a finite Lüroth expansion. We conjecture that this also applies to fractions of the form $\frac a{3^k}$. It would be very interesting to get a full understanding of rationals with a finite Lüroth expansion.

\section*{Acknowledgment}
This paper was written during an internship at the
Max~Planck~Institute for Mathematics in Bonn. We would like to thank the
MPIM for offering us this unique opportunity to conduct active mathematical
research. We are very grateful to our adviser, Izabela Petrykiewicz, for her
helpful suggestions and support.

\nocite*
\bibliographystyle{amsalpha}
\bibliography{references}
\end{document}